\documentclass[preprint,12pt]{elsarticle}

\usepackage{amssymb}
\usepackage{amsthm}
\usepackage{amsmath}

\newcommand{\sysn}{\left\{\begin{array}{rcl}}
\newcommand{\sysk}{\end{array}\right.}

\newcommand{\markwin}{\underset{mark}{\uparrow}}
\newcommand{\tuple}[1]{\left\langle #1\right\rangle}

\newtheorem{theorem}{Theorem}[section]
\newtheorem{lemma}[theorem]{Lemma}

\theoremstyle{example}

\newtheorem{proposition}[theorem]{Proposition}
\theoremstyle{definition}
\newtheorem{definition}[theorem]{Definition}

\newtheorem{question}[theorem]{Question}
\newtheorem{corollary}[theorem]{Corollary}

\journal{...}

\begin{document}

\begin{frontmatter}



\title{On Ramsey properties, function spaces, and topological games}


\author[label1]{Steven Clontz}

\ead[label1]{sclontz@southalabama.edu}


\address[label1]{Department of Mathematics and Statistics, \\
University of South Alabama, Mobile, Alabama, U.S.A.}

\author[label2]{Alexander V. Osipov}

\ead[label2]{OAB@list.ru}


\address[label2]{Krasovskii Institute of Mathematics and Mechanics, Ural Federal
 University, \\ Ural State University of Economics, Yekaterinburg, Russia}

\begin{abstract}
An open question of Gruenhage asks if all strategically selectively separable
spaces are Markov selectively separable, a game-theoretic
statement known to hold for countable spaces. As a corollary of a result by
Berner and Juh\'asz, we note that the ``strong''
version of this statement,
where the second player is restricted to selecting single points in the
rather than finite subsets,
holds for all $T_3$ spaces without isolated points.
Continuing this investigation,
we also consider games related to selective sequential separability, and demonstrate
results analogous to those for selective separability. In particular,
strong selective sequential separability in the presence of the Ramsey property
may be reduced to a weaker condition on a countable sequentially dense subset.
Additionally, \(\gamma\)- and \(\omega\)-covering properties on \(X\) are shown
to be equivalent to corresponding sequential properties on \(C_p(X)\).
A strengthening of the Ramsey property is also introduced, which is still equivalent to
\(\alpha_2\) and \(\alpha_4\) in the context of \(C_p(X)\).
\end{abstract}


\end{frontmatter}



\section{Introduction}

Let $\mathcal{A}$ and $\mathcal{B}$ be sets whose elements are
families of subsets of an infinite set $X$. Then $S_1(\mathcal{A},
\mathcal{B})$ denotes a \textit{selection principle}: for each
sequence $(A_n : n\in \omega)$ of elements of $\mathcal{A}$ there
is a sequence $(b_n: n\in \omega)$ such that for each $n$, $b_n\in
A_n$, and $\{b_n: n\in \omega\}$ is an element of $\mathcal{B}$.

$S_{fin}(\mathcal{A},\mathcal{B})$ is a \textit{selection
principle}: for each sequence $(A_{n}: n\in \omega)$ of elements
of $\mathcal{A}$ there is a sequence $(B_{n} : n\in \omega)$ of
finite sets such that for each $n$, $B_{n}\subseteq A_{n}$, and
$\bigcup_{n\in\omega}B_{n}\in\mathcal{B}$.

In this paper, by a cover we mean a nontrivial one; that is,
$\mathcal{U}$ is a cover of $X$ if $X=\bigcup \mathcal{U}$ and
$X\notin \mathcal{U}$.

 A cover $\mathcal{U}$ of a space $X$ is:

\begin{itemize}
\item an {\it $\omega$-cover} if every finite subset of $X$ is contained in a
 member of $\mathcal{U}$.

\item a {\it $\gamma$-cover} if it is infinite and each $x\in
X$ belongs to all but finitely many elements of $\mathcal{U}$.
\end{itemize}

Note that every $\gamma$-cover contains a countable
$\gamma$-cover, and every $\gamma$-cover is also an $\omega$-cover.

For a topological space $X$ we denote:

\begin{itemize}
\item $\Omega$ --- the family of all open $\omega$-covers of
$X$;

\item $\Gamma$ --- the family of all open $\gamma$-covers of
$X$.
\end{itemize}

Let $X$ be a Hausdorff topological space, and $x\in X$.
A subset $A$ of $X$
{\it converges} to a unique $x=\lim A$ if $A$ is infinite, $x\notin
A$, and for each neighborhood $U$ of $x$, $A\setminus U$ is
finite; We also assume $x=\lim\{x\}$.
We may then consider the following collections:

\begin{itemize}
\item $\Omega_x=\{A\subseteq X : x\in \overline{A}\setminus
A\text{ or }A=\{x\}\}$;

\item $\Gamma_x=\{A\subseteq X : x=\lim A\}$.
\end{itemize}

Note that if $A\in \Gamma_x$, then there exists
a countable set $A'=\{a_n:n<\omega\}\subseteq A$ with
$A'\in \Gamma_x$. As such, $\Gamma_x$ may be considered
to be the set of non-trivial convergent sequences to $x$.

As was noted earlier, $\Gamma\subseteq\Omega$; likewise,
$\Gamma_x\subseteq\Omega_x$.

Given these definitions, we may describe the following
well-known selection principles.

\begin{itemize}
\item A space $X$ has Arhangel'skii's {\it countable fan tightness}
if $X$ satisfies
$S_{fin}(\Omega_x,\Omega_x)$
for every \(x\in X\) \cite{arh}.

\item A space $X$ has Sakai's {\it countable strong fan tightness}
if $X$ satisfies
$S_{1}(\Omega_x,\Omega_x)$
for every \(x\in X\)  \cite{sak}.

\item A space $X$ has Arhangel'skii's {\it property $\alpha_4$},
if $X$ satisfies $S_{fin}(\Gamma_x,\Gamma_x)$
for every \(x\in X\)  \cite{arh0}.

\item  A space $X$ has Arhangel'skii's {\it property $\alpha_2$},
if $X$ satisfies $S_{1}(\Gamma_x,\Gamma_x)$
for every \(x\in X\)  \cite{arh0}.

\item A space $X$ is {\it strictly
Fr$\acute{e}$chet-Urysohn} if $X$ satisfies
$S_{1}(\Omega_x,\Gamma_x)$
for every \(x\in X\)  \cite{sash}.

\item A space $X$ is {\it strongly
Fr$\acute{e}$chet-Urysohn}
if $X$ satisfies
$S_{fin}(\Omega_x,\Gamma_x)$
for every \(x\in X\)  \cite{mic,siw}.

\end{itemize}

It is easy to check that $X$ satisfies $S_{fin}(\Gamma_x,
\Omega_x)$ for any $x\in X$ if and only if $X$ does not contain
a copy of the \textit{sequential fan} $S_{\omega}$,
where $S_{\omega}$ is the quotient space of
countably many convergent sequences obtained by identifying all
limit points.

\begin{definition}[\cite{nosh}]
A space $X$ has the {\it Ramsey property} if for any choices
$x_{i,j}\in X$ for $i,j\in \omega$ such that $\lim\{\lim\{
x_{i,j}:j\in\omega\}:i\in\omega\}=x$ for some point $x\in X$,
there exists an infinite set $M\subseteq \omega$ such that for
every open neighborhood $U$ of $x$, $x_{m,n}\in U$ for
sufficiently large $m,n\in M$ with $m<n$.
\end{definition}

In particular, note that $x=\lim\{x_{m,m^+}:m\in M\}$ where
$m^+=\min(\{k\in M:k>M\})$, and thus Ramsey $\Rightarrow\alpha_4$
(and furthermore $\alpha_3$; see \cite{nosh}). But the relation between
$\alpha_2$ and the Ramsey property
remains open, even for topological groups (Question 3.15 in \cite{shak}).

We also will use the following strengthening of Ramsey:

\begin{definition}
A space $X$ has the {\it $\Omega$-Ramsey property} if and only if
for any choices $T_{i,j}\in[X]^{<\omega}$ for $i,j\in\omega$ such
that $\lim\{\lim\bigcup_{j\in\omega}T_{i,j}:i\in\omega\}=x$ for
some point $x\in X$, there exists an infinite set $M\subseteq
\omega$ such that for every open neighborhood $U$ of $x$,
$T_{m,n}\subseteq U$ for sufficiently large $m<n\in M$.
\end{definition}

The following implications follow for any topological space $X$
since \(\Gamma_x\subseteq\Omega_x\):

\begin{center}
$S_1(\Gamma_x,\Gamma_x) \Rightarrow S_{fin}(\Gamma_x,\Gamma_x)
\Rightarrow S_{fin}(\Gamma_x,\Omega_x) \Leftarrow
S_{1}(\Gamma_x,\Omega_x)$ \\  $\Uparrow$ \, \,\, \, \, \, \, \,\,
\, \, \, \, \, $ \Uparrow $ \,\, \, \, \,\, \, \, \, \, \,
$\Uparrow $ \,\, \, \, \,\,\,\, \, \, \, \,\, \, \, $\Uparrow$ \\
$S_1(\Omega_x,\Gamma_x) \Rightarrow S_{fin}(\Omega_x,\Gamma_x)
\Rightarrow S_{fin}(\Omega_x,\Omega_x) \Leftarrow
S_{1}(\Omega_x,\Omega_x)$

\end{center}

 If $X$ is a space and $A\subseteq X$, then the sequential closure of $A$,
 denoted by $[A]_{seq}$, is the set of all limits of sequences
 from $A$. A set $D\subseteq X$ is said to be sequentially dense
 if $X=[D]_{seq}$. A space $X$ is called sequentially separable if
 it has a countable sequentially dense set.

For a topological space $X$ we denote:
\begin{itemize}
\item $\mathcal{D}$ is the family of all dense subsets of $X$;

\item $\mathcal{S}$ is the family of all sequentially dense
subsets of $X$.
\end{itemize}

Let $\Pi$ represent $S_1$ or $S_{fin}$.
When we write $\Pi (\mathcal{A}, \mathcal{B}_x)$ without specifying
$x$, we mean $(\forall x) \Pi (\mathcal{A}, \mathcal{B}_x)$.
\medskip

As above,
the following implications hold on any topological space $X$
since \(\mathcal S\subseteq\mathcal D\):

\begin{center}
$S_1(\mathcal{S},\Gamma_x) \Rightarrow
S_{fin}(\mathcal{S},\Gamma_x) \Rightarrow
S_{fin}(\mathcal{S},\Omega_x) \Leftarrow
S_{1}(\mathcal{S},\Omega_x)$ \\  \, \, \, $\Uparrow$ \, \, \, \,
\,\,  \, \, \, \, \, $ \Uparrow $ \,\, \, \, \,\, \, \, \, \, \,
$\Uparrow $ \,\, \, \, \,\, \, \,\, \, \, $\Uparrow$
\\ $S_1(\mathcal{D},\Gamma_x) \Rightarrow
S_{fin}(\mathcal{D},\Gamma_x) \Rightarrow
S_{fin}(\mathcal{D},\Omega_x) \Leftarrow
S_{1}(\mathcal{D},\Omega_x)$

\end{center}

Some of these selection principles are known by name.

\begin{itemize}
\item A space $X$ is {\it $R$-separable}, if $X$ satisfies
$S_1(\mathcal{D}, \mathcal{D})$ (Def. 47, \cite{bbm1}).

\item A space $X$ is {\it $M$-separable} (or selectively
separable), if $X$ satisfies $S_{fin}(\mathcal{D}, \mathcal{D})$.

\item A space $X$ is {\it selectively sequentially separable}, if
$X$ satisfies $S_{fin}(\mathcal{S}, \mathcal{S})$ (Def. 1.2,
\cite{bc}).
\end{itemize}

\begin{proposition}[Proposition 1.3 in \cite{bc}] Every
sequentially dense subspace of a selectively sequentially
separable space is sequentially separable. In particular, every
selectively sequentially separable space is sequentially
separable.
\end{proposition}

And so the following implications hold on any topological space $X$:

\begin{center}
$S_1(\mathcal{S},\mathcal{S}) \Rightarrow
S_{fin}(\mathcal{S},\mathcal{S}) \Rightarrow
S_{fin}(\mathcal{S},\mathcal{D}) \Leftarrow
S_{1}(\mathcal{S},\mathcal{D})$ \\  \, \, $\Uparrow$ \, \, \, \,
\,\, \,  \, \, \, $ \Uparrow $ \,\, \, \, \,\, \, \,
\, \, \, $\Uparrow $ \,\, \, \, \,\, \, \,\, \, \, $\Uparrow$ \\
$S_1(\mathcal{D},\mathcal{S}) \Rightarrow
S_{fin}(\mathcal{D},\mathcal{S}) \Rightarrow
S_{fin}(\mathcal{D},\mathcal{D}) \Leftarrow
S_{1}(\mathcal{D},\mathcal{D})$

\end{center}

We now have three types of topological properties
described as selection principles:

\begin{itemize}
\item  local properties of the form $S_*(\Phi_x,\Psi_x)$;

\item  semi-local properties of the form $S_*(\Phi,\Psi_x)$.

\item  global properties of the form $S_*(\Phi,\Psi)$;
\end{itemize}

 There is a game, denoted by $G_{fin}(\mathcal{A},\mathcal{B})$,
 corresponding to $S_{fin}(\mathcal{A},\mathcal{B})$.
In this game two players,
 ONE and TWO, play a round for each natural number $n$. In the
 $n$-th round ONE chooses a set $A_n\in \mathcal{A}$ and TWO
 responds with a finite subset $B_n$ of $A_n$. A play
 $A_1,B_1;...;A_n,B_n;...$ is won by TWO if $\bigcup\limits_{n\in
 \omega} B_n\in \mathcal{B}$; otherwise, ONE wins.
Similarly, one defines the game $G_{1}(\mathcal{A},\mathcal{B})$,
associated with $S_{1}(\mathcal{A},\mathcal{B})$.

A strategy of a player is a function $\sigma$ from the set of all
finite sequences of moves of the opponent into the set of (legal)
moves of the strategy owner.

It then follows that the selection principle
$S_*(\mathcal A,\mathcal B)$ is equivalent to player ONE
lacking a winning \textit{predetermined}
strategy for $G_*(\mathcal A,\mathcal B)$ that
is defined solely on the current round number $n$ (ignoring the
moves of TWO) \cite{cl2}. Even when ONE lacks such a predetermined
winning strategy, it is still possible
for ONE to have a winning strategy that uses perfect information.

As such, we now have three types of topological games on a topological space
$X$:

\begin{itemize}
\item  local games of the form $G_*(\Phi_x,\Psi_x)$;

\item  semi-local games of the form $G_*(\Phi,\Psi_x)$.

\item  global games of the form $G_*(\Phi,\Psi)$;
\end{itemize}

Let us now more formally define our ``strategies''.

\begin{definition} A {\it strategy} for TWO in the game
$G_{fin}(\mathcal{A},\mathcal{B})$ is a function $\sigma$
satisfying $\sigma(\tuple{A_0,...,A_n})\in [A_n]^{<\omega}$ for
$\tuple{A_0,...,A_n}\in \mathcal{A}^{n+1}$. We say this strategy
is {\it winning} if whenever ONE plays $A_n\in \mathcal{A}$ during
each round $n<\omega$, TWO wins the game by playing
$\sigma(\tuple{A_0,...,A_n})$ during each round $n<\omega$. If a
winning strategy exists, then we write TWO $\uparrow
G_{fin}(\mathcal{A},\mathcal{B})$.
\end{definition}

We will also be interested in strategies that use limited information;
specifically, those that only use the current round number $n$
and the most recent move of the opponent.

\begin{definition} A {\it Markov strategy} for TWO in the game
$G_{fin}(\mathcal{A},\mathcal{B})$ is a function $\sigma$
satisfying $\sigma(A,n)\in [A_n]^{<\omega}$ for $A\in \mathcal{A}$
and $n\in \omega$. We say this Markov strategy is {\it
winning} if whenever ONE plays $A_n\in \mathcal{A}$ during each
round $n<\omega$, TWO wins the game by playing $\sigma(A_n,n)$
during each round $n<\omega$. If a winning Markov strategy exists,
then we write TWO ${\markwin}
G_{fin}(\mathcal{A},\mathcal{B})$.
\end{definition}

Both definitions may be naturally modified for the game
$G_1(\mathcal A,\mathcal B)$ instead. It is then easily
seen that \[\text{TWO} {\markwin}
G_{\ast}(\mathcal{A},\mathcal{B}) \Rightarrow \text{TWO} \uparrow
G_{\ast}(\mathcal{A},\mathcal{B}) \Rightarrow
S_{\ast}(\mathcal{A},\mathcal{B})\] where $\ast\in \{1,fin\}$.

\section{Main results}

Barman and Dow showed (\cite{bd}, Theorem 2.9) that every
separable Fr$\acute{e}$chet-Urysohn $T_2$-space is selectively
separable. By definition of Fr$\acute{e}$chet-Urysohn,
closure is equivalent to sequential closure in such spaces, so we
immediately have:

\begin{proposition}(Proposition 2.2. in \cite{bc}) Every Fr$\acute{e}$chet-Urysohn
separable $T_2$-space is selectively sequentially separable.
\end{proposition}

Let $\Gamma_x'=\{A\subseteq X:\exists B\in\Gamma_x(B\subseteq
A)\}$, and note that $\mathcal S\subseteq\Gamma_x'$ (while
$\mathcal S\not\subseteq\Gamma_x$). These may be considered the
sequences which cluster at $x$.

In particular, we have that \(S_*(\Phi,\mathcal S)\Rightarrow
S_*(\Phi,\Gamma_x')\) (with similar game-theoretic results).
We now turn to the following theorem:

\begin{theorem}\label{th2}
Let $*\in\{1,fin\}$; if $*=1$ assume $X$ is Ramsey,
and otherwise assume $X$ is $\Omega$-Ramsey.
Then for any non-empty set $\Phi$, the following are equivalent:

\begin{enumerate}

\item $X$ satisfies $S_{*}(\Phi,\mathcal{S})$ $($resp. TWO
$\uparrow G_{*}(\Phi,\mathcal{S})$, TWO ${\markwin}
G_{*}(\Phi,\mathcal{S}))$;

\item $X$ is sequentially separable and satisfies $S_{*}(\Phi,
\Gamma_{x}')$ $($resp. TWO $\uparrow G_{*}(\Phi, \Gamma_{x}')$, TWO
${\markwin} G_{*}(\Phi, \Gamma_{x}'))$;

\item $X$ has a countable sequentially dense subset $D$ where
$S_{*}(\Phi, \Gamma_{x}')$  $($resp. TWO $\uparrow G_{*}(\Phi,
\Gamma_{x}')$, TWO ${\markwin} G_{*}(\Phi, \Gamma_{x}'))$
holds for all $x\in D$.

\end{enumerate}

\end{theorem}

\begin{proof}

Let $P\in\Phi$. Then for the countable set $\{P\}$, we may
apply any variant of the first condition to obtain
$T_i\in [P]^{<\omega}$
for $i\in\omega$ with
$\bigcup\{T_i:i\in\omega\}\in\mathcal S$, demonstrating the
respective second condition, which trivially implies the third.
As such, we only need prove that the final condition implies
the first; let $D=\{d_i:i\in\omega\}$ witness that final condition.

a) Assume $S_*(\Phi,\Gamma_{x}')$ for $x\in D$.
Let $P_{i,m}\in\Phi$ for all $i,m\in\omega$.
For each $i\in\omega$, $S_*(\Phi,\Gamma_{d_i}')$ allows us to choose
$T_{i,m}\in [P_{i,m}]^{*}$ and $m_t\in\omega$ for $t\in\omega$
such that $d_i=\lim\bigcup\{T_{i,m_t}:t\in\omega\}$.
We claim that $\bigcup\{T_{i,m}:i,m\in\omega\}$ is sequentially dense.
To see this, let $x\in X$, and choose $i_s\in \omega$ for $s\in\omega$
such that $x=\lim\{d_{i_s}:s\in\omega\}$. We then choose $M\subseteq\omega$
witnessing the appropriate Ramsey property
for $\{T_{i_s,m_t}:s,t\in\omega\}$ and
$x$; it follows that $x=\lim\bigcup\{T_{i_s,m_{s^+}}:s\in M\}$.
Thus for any countable collection of sets $P_{i,m}\in\Phi$, we have
$T_{i,m}\in [P_{i,m}]^*$ with
$\bigcup\{T_{i,m}:i,m\in\omega\}$ sequentially
dense, witnessing $S_1(\Phi,\mathcal S)$.

b) Now assume $TWO\uparrow G_*(\Phi,\Gamma_{d_i}')$ is witnessed by
the strategy $\sigma_i$ for each $i\in\omega$.
Let $p:\omega\rightarrow\omega$ be a function such that
$p^{\leftarrow}(i)$ is infinite for all $i\in\omega$. For a nonempty finite
sequence $t$, let $t'$ be its subsequence removing all terms of
index $n$ such that $p(n)\neq p(|t|-1)$.
We define the strategy $\sigma$ for the game
$G_{*}(\mathcal{S},\mathcal{S})$ by
$\sigma(t)=\sigma_{p(|t|-1)}(t')$; that is, $\sigma$ partitions
any counterplay by ONE into countably many subplays according to $p$,
and uses a different $\sigma_i$ for each subplay.

Let $\alpha\in \mathcal{S}^{\omega}$, and let $\alpha_i$ be its
subsequence removing all terms of index $n$ such that $p(n)\neq
i$. Then $\bigcup\{\sigma_i(\alpha_i \upharpoonright (n+1)): n\in
\omega\}\in \Gamma_{d_i}'$ since $\sigma_i$ is a winning strategy
for TWO, so choose $n_{i,t}\in\omega$ for $t\in\omega$ where
$d_i=\lim\bigcup
\{\sigma_i(\alpha_i \upharpoonright (n_t+1)):t\in\omega\}$.

We claim that
$\bigcup\{\sigma(\alpha \upharpoonright (n+1)): n\in \omega\}\in
\mathcal{S}$, so let $x\in X$. Then there exists
$\{d_{i_s} : s\in\omega\}$ such that  $x=\lim\{d_{i_s}:s\in\omega\}$.
We then apply the appropriate Ramsey property to
$\{\sigma_{i_s}(\alpha_{i_s} \upharpoonright (n_{i_s,t}+1)): s,t\in
\omega\}$ to obtain an $M\subseteq\omega$ with
$x=\lim\{\sigma_{i_s}(\alpha_{i_s}
\upharpoonright(n_{i_s,s^+}+1)):s\in M\}$.
Since each
$
\sigma_{i_s}(\alpha_{i_s}\upharpoonright(n_{i_s,s^+}+1))
  =
\sigma(\alpha\upharpoonright (n+1))
$
for some $n\in\omega$, the result follows.

c) Finally let TWO ${\markwin} G_{1}(\mathcal{S}, \Gamma_{d_i})$
for each $i\in\omega$ be witnessed by $\sigma_i$.
Let $p:\omega\rightarrow\omega$ be a function such that
$p^{\leftarrow}(i)$ is infinite for all $i\in\omega$.
We then define the Markov strategy $\sigma$ by
\[\sigma(P,n)=\sigma_{p(n)}(P, |\{m<n: p(m)=p(n)\}|)\]
so that as in the previous case, $\sigma$ partitions
any counterplay by ONE into countably many subplays according to $p$,
and uses a different $\sigma_i$ for each subplay.

Let $\alpha\in\mathcal{S}^\omega$, and let $\alpha_i$ be its
subsequence removing all terms of index $n$ such that $p(n)\neq i$.
Then $\{\sigma_i(\alpha_i(n),n): n\in \omega\}\in \Gamma_{d_i}'$
since $\sigma_i$ is a winning strategy
for TWO, so choose $n_{i,t}\in\omega$ for $t\in\omega$ where
$d_i=\lim\{\sigma_i(\alpha_i(n_{i,t}),n_{i,t}):t\in\omega\}$.

We claim that
$\{\sigma(\alpha(n),n): n\in \omega\}\in
\mathcal{S}$, so let $x\in X$. Then there exists
$\{d_{i_s} : s\in\omega\}$ such that  $x=\lim\{d_{i_s}:s\in\omega\}$.
We then apply the appropriate Ramsey property to
$\{\sigma_{i_s}(\alpha_{i_s}(n_{i_s,t}),n_{i_s,t}):s,t\in\omega\}$
to obtain an $M\subseteq\omega$ with
$x=\lim\{\sigma_{i_s}(\alpha_{i_s}(n_{i_s,s^+}),n_{i_s,s^+}):s\in M\}$.
Since each
$
\sigma_{i_s}(\alpha_{i_s}(n_{i_s,s^+}),n_{i_s,s^+})
  =
\sigma(\alpha(n),n)
$
for some $n\in\omega$, the result follows.
\end{proof}

The previous result mirrors the following slight generalization
of theorems 16 and 41 of \cite{cl}.

\begin{theorem}[\cite{cl}]\label{th1}\label{th5}
For a topological space $X$, nonempty set $\Phi$, and $*\in \{1,fin\}$,
the following are equivalent:

\begin{enumerate}
\item $X$ satisfies $S_{*}(\Phi,\mathcal{D})$ $($resp. TWO
$\uparrow G_{*}(\Phi,\mathcal{D})$, TWO
${\markwin} G_{*}(\Phi,\mathcal{D}))$;

\item $X$ is separable and  satisfies $S_{*}(\Phi,
\Omega_{x})$ $($resp. TWO $\uparrow G_{*}(\Phi,
\Omega_{x})$, TWO ${\markwin} G_{*}(\Phi,
\Omega_{x}))$;

\item $X$ has a countable dense subset $D$ where
$S_{*}(\Phi, \Omega_{x})$  $($resp. TWO $\uparrow
G_{*}(\Phi, \Omega_{x})$, TWO ${\markwin}
G_{*}(\Phi, \Omega_{x}))$ holds for all $x\in D$.
\end{enumerate}
\end{theorem}
\begin{proof}
In \cite{cl}, $\Phi=\mathcal D$ was an additional assumption,
but was never required in the proofs, since
$S_*(\Phi,\mathcal D)$ implies separability for any non-empty
$\Phi$.
\end{proof}

Recall that a $\pi$-base for a space $X$ is a family
$\mathcal{U}$ of nonempty open subsets of
$X$ such that for each nonempty open set $V\subseteq X$ there is a
$U\in \mathcal{U}$ with $U\subseteq V$.
Then the $\pi$-weight of a space $X$, denoted $\pi(X)$, is the minimal
cardinality of a $\pi$-base for $X$.

\begin{corollary}\label{th7} Let $X$ be a $T_3$-space with no isolated
points. Then the following  are equivalent:

\begin{enumerate}

\item $\pi(X)=\aleph_0$;

\item  TWO $\uparrow
G_1(\mathcal{D},\mathcal{D})$;

\item  TWO ${\markwin}
G_1(\mathcal{D},\mathcal{D})$;

\item  $X$ is separable and TWO $\uparrow
G_1(\mathcal{D},\Omega_x)$;

\item  $X$ is separable and TWO ${\markwin}
G_1(\mathcal{D},\Omega_x)$;

\item $X$ has a countable dense subset $D$ where TWO $\uparrow
G_{1}(\mathcal{D}, \Omega_{x})$ for all $x\in D$.

\item $X$ has a countable dense subset $D$ where TWO ${\markwin}
G_{1}(\mathcal{D}, \Omega_{x})$ for all $x\in D$.

\end{enumerate}

\end{corollary}
\begin{proof}
The equivalence of (1) and (2) is \cite[Theorem 2.1]{beju}.

Assuming (1), let $\{P_n:n\in\omega\}$ be a countable $\pi$-base.
We may then define $\sigma(D,n)\in D\cap P_n$ arbitrarily, and
it's easy to see that this is winning for TWO, implying (3)
and therefore (2).

All other equivalencies follow from Theorem \ref{th5}.
\end{proof}

The equivalance (2) $\Leftrightarrow$ (3) is similar to the following
open question of Gruenhage, first shown to be true when $X$ is countable
by Barman and Dow in
\cite[Theorem 2.11]{bd2}; see \cite[Lemma 37]{cl} for a
general sufficient condition which guarantees that a winning strategy
may be improved to a Markov winning strategy.

\begin{question}
When does TWO $\uparrow G_{fin}(\mathcal D,\mathcal D)$ imply
TWO ${\markwin} G_{fin}(\mathcal D,\mathcal D)$?
\end{question}

\section{$\Omega$-Ramsey in Topological Groups}

We now adapt techniques of Sakai \cite{sak3} to obtain the following lemma
giving a useful recharacterization of the $\Omega$-Ramsey property for
topological groups, which we require in the following section.

\begin{lemma}\label{lem5}
Let $\tuple{G,\cdot}$ be a topological group with unit $e$.
Then the $\Omega$-Ramsey property is equivalent to the following:
if $T_{n,m}\in[G]^{<\omega}$ and $e=\lim\bigcup\{T_{n,m}:m\in\omega\}$
for each $n\in\omega$, then there exists $M\subseteq\omega$
such that $e=\lim\bigcup\{T_{n,m}:n,m\in M,n<m\}$.
\end{lemma}

\begin{proof}
The forward direction follows by noting that $e=\lim\{e\}$ and
thus applying the $\Omega$-Ramsey property to
$\{T_{n,m}:n,m\in\omega\}$.

For the converse, let $x_n=\lim\bigcup\{T_{n,m}:m\in\omega\}$ for
each $n\in\omega$, and $e=\lim\{x_n:n\in\omega\}$ (since $G$ is
homoegeneous). If $S_{n,m}=x_n^{-1}\cdot T_{n,m}$, it follows
that $\lim\bigcup\{S_{n,m}:m\in\omega\}=x_n^{-1}\cdot x_n=e$.
We apply the assumption to obtain $M\subseteq\omega$ where
$e=\lim\bigcup\{S_{n,m}:n,m\in M,n<m\}$, and claim that $M$
witnesses $\Omega$-Ramsey.

Let $U$ be a neighborhood of $e$, which must contain
$\{x_n:n\geq k'\}$ for some $k'\in\omega$. By applying
\cite[Lemma 2.3]{sak3}, we may choose an open neighborhood $V$
of $e$ where $\{x_n:n\geq k'\}\cdot V\subseteq U$. Since
$e=\lim\bigcup\{S_{n,m}:n,m\in M,n<m\}$, we may choose $k\geq k'$
where $\bigcup\{S_{n,m}:n,m\in M,k\leq n<m\}\subseteq V$. So for $k\leq n<m$,
\[S_{n,m}\subseteq V \Rightarrow
T_{n,m}=x_n\cdot S_{n,m}\subseteq x_n\cdot V\subseteq U\]
\end{proof}

\section{Applications in $C_p$-theory}

For a Tychonoff space $X$, we denote by $C_p(X)$ the topological group
of all real-valued continuous functions on $X$ with the topology of
pointwise convergence. The symbol $\bf{0}$ stands for the constant
function to $0$.

Basic open sets of $C_p(X)$ are of the form
$[x_1,...,x_k; U_1,...,U_k]=\{f\in C_p(X): f(x_i)\in U_i$,
$i=1,...,k\}$, where each $x_i\in X$ and each $U_i$ is a non-empty
open subset of $\mathbb{R}$. When $U_i=U$ for all $i\leq k$,
we simply write $[x_1,\dots,x_k;U]$.

Consider the following result of Sakai.

\begin{theorem}[Theorem 2.5 of \cite{sak3}]
The Ramsey property is equivalent to $\alpha_2$
and $\alpha_4$ for $C_p(X)$.
\end{theorem}

By using the previous Lemma \ref{lem5}, we may show the following.

\begin{theorem}\label{th52}\label{th51}
The $\Omega$-Ramsey property is equivalent to the Ramsey,
$\alpha_2$, and $\alpha_4$ properties for $C_p(X)$.
\end{theorem}
\begin{proof}
Let $T_{n,m}\in[C_p(X)]^{<\omega}$ and
${\bf0}=\lim\bigcup\{T_{n,m}:m\in\omega\}$ for each $n\in\omega$.
We let $g_{n,m}(x)=\max\{|f(x)|:f\in \bigcup_{i\leq n}T_{i,m}\}$,
noting ${\bf0}=\lim\{g_{n,m}:m\in\omega\}$ for each $n\in\omega$.
We apply $\alpha_2$, that is, $S_1(\Gamma_{\bf0},\Gamma_{\bf0})$
to $\{g_{n,m}:n<m\in\omega\}$ to obtain an increasing mapping
$\phi:\omega\to\omega$ with
${\bf0}=\lim\{g_{m,\phi(m)}:m\in\omega\}$.

Now let $\phi^0(n)=n$ and $\phi^{i+1}(n)=\phi(\phi^i(n))$ and set
$M=\{\phi^i(0):i\in\omega\}$. We will demonstrate that
${\bf0}=\lim\{T_{n,m}:n,m\in M,n<m\}$. For $x\in X$ and
$\epsilon>0$, pick $k\in\omega$ where
$|g_{m,\phi(m)}(x)|<\epsilon$ for $k<m,m\in M$. It follows that
for $f\in T_{n,m}$ where $n,m\in M$ and $k<n<m$, let $m=\phi(m')$
where $n\leq m'$. Then $|f(x)|\leq|g_{n,m}(x)|\leq|g_{m',m}(x)|
=|g_{m',\phi(m')}(x)|<\epsilon$.
Thus \(C_p(X)\) is \(\Omega\)-Ramsey.

Since \(\Omega\)-Ramsey implies Ramsey, the result follows.
\end{proof}

Recall that the $i$-weight $iw(X)$ of a space $X$ is the smallest
infinite cardinal number $\tau$ such that $X$ can be mapped by a
one-to-one continuous mapping onto a Tychonoff space of the weight
not greater than $\tau$.

\begin{theorem}[Noble \cite{nob}] \label{th31}
Let $X$ be a   space. A space $C_{p}(X)$ is separable if and only if
$iw(X)=\aleph_0$.
\end{theorem}

Note that if $X$ is itself Tychonoff and $iw(X)=\aleph_0$,
then the image of $X$ under a witnessing
one-to-one continuous mapping yields a coarser topology for $X$
which is separable and metrizable; this is the characterization
given in \cite{mc}.

In papers \cite{arh,arh2,bbm1,koc,os1,os2,os4,os10,sak,sch1}
various selection principles for a Tychonoff space $X$ were
related to the selection principles for $C_p(X)$. Likewise, in
\cite{cl,ljos,sak,sch3,sch2} various selection games for $X$ and
$C_p(X)$  and a bitopological space $(C(X),\tau_k, \tau_p)$ were
related.

So we have the following applications in $C_p$-theory.

\begin{theorem}[Theorems 22 and 43 in \cite{cl}]\label{th112}
For a Tychonoff space $X$ and $*\in \{1,fin\}$, the
following are equivalent:

\begin{enumerate}

\item TWO has a winning (Markov) strategy in $G_{*}(\Omega,
\Omega)$ on $X$;

\item TWO has a winning (Markov) strategy in $G_{*}$
$(\Omega_{\bf0},\Omega_{\bf0})$ on $C_p(X)$;

\item TWO has a winning (Markov) strategy in
$G_{*}(\mathcal{D},\Omega_{\bf0})$ on $C_p(X)$.

\end{enumerate}

\end{theorem}

\begin{corollary}\label{th33} Let $X$ be a Tychonoff space with
a coarser second-countable topology
(that is, \(iw(X)=\aleph_0\)) and $*\in
\{1,fin\}$. The following assertions are equivalent:

\begin{enumerate}

\item TWO has a winning (Markov) strategy in $G_{*}(\Omega,
\Omega)$ on $X$;

\item TWO has a winning (Markov) strategy in $G_{*}$
$(\Omega_{\bf0},\Omega_{\bf0})$ on $C_p(X)$;

\item TWO has a winning (Markov) strategy in
$G_{*}(\mathcal{D},\Omega_{\bf0})$ on $C_p(X)$.

\item TWO has a winning (Markov) strategy in
$G_{*}(\mathcal{D},\mathcal{D})$ on $C_p(X)$;

\end{enumerate}

\end{corollary}

\begin{proof}
By Theorems \ref{th1}, \ref{th31} and \ref{th112},
 the items (1-4) are equivalent.
\end{proof}

\begin{corollary} Let $X$ be a Tychonoff space with a coarser
second-countable topology. The following assertions are equivalent:

\begin{enumerate}

\item $\pi(C_p(X))=\aleph_0$;

\item  TWO $\uparrow$
$G_1(\mathcal{D},\mathcal{D})$ for $C_p(X)$;

\item TWO $\uparrow$ $G_1(\mathcal{D},\Omega_{x})$
for $C_p(X)$;

\item TWO $\uparrow$ $G_{1}(\Omega, \Omega)$ for
$X$;

\item  TWO ${\markwin}$
$G_1(\mathcal{D},\mathcal{D})$ for $C_p(X)$;

\item TWO ${\markwin}$
$G_1(\mathcal{D},\Omega_{x})$ for $C_p(X)$;

\item TWO ${\markwin}$ $G_{1}(\Omega, \Omega)$
for $X$;

\item $X$ is countable.

\end{enumerate}

\end{corollary}

\begin{proof}
Items (1-7) follow from Theorem \ref{th7} and Corollary \ref{th33}.
The fact that (8) is equivalent to (6) and (7) doesn't require $iw(X)=\aleph_0$
and may be found in \cite[Theorem 17]{cl3}
along with several other equivalencies.
\end{proof}

We now turn to the case where TWO may choose finite sets
each round.

\begin{corollary} Let $X$ be a separable metrizable space.
Then the following  are equivalent:

\begin{enumerate}

\item  TWO $\uparrow$
$G_{fin}(\mathcal{D},\mathcal{D})$ for $C_p(X)$;

\item  TWO $\uparrow$
$G_{fin}(\mathcal{D},\Omega_{x})$ for $C_p(X)$;

\item TWO $\uparrow$ $G_{fin}(\Omega, \Omega)$ for
$X$;

\item  TWO ${\markwin}$
$G_{fin}(\mathcal{D},\mathcal{D})$ for $C_p(X)$;

\item TWO ${\markwin}$
$G_{fin}(\mathcal{D},\Omega_{x})$ for $C_p(X)$;

\item TWO ${\markwin}$ $G_{fin}(\Omega,
\Omega)$ for $X$;

\item $X$ is $\sigma$-compact.

\end{enumerate}

\end{corollary}

\begin{proof}
Second-countability allows us to
apply Corollary \ref{th33} to show (1-3) are mutually equivalent,
as are (4-6).
By \cite[Corollary 39]{cl}, (3) is equivalent to (6),
and by \cite[Lemma 24]{cl}, (6) equivalent to (7).
\end{proof}

We now demonstrate analogous results, replacing $\mathcal D$
and $\Omega$ with $\mathcal S$ and $\Gamma$.

\medskip

We recall that a subset of $X$ that is the
 complete preimage of zero for a certain function from~$C(X)$ is called a zero-set.
A subset $O\subseteq X$  is called  a cozero-set (or functionally
open) of $X$ if $X\setminus O$ is a zero-set.

A $\gamma$-cover $\mathcal{U}$ of co-zero sets of $X$  is {\it
$\gamma_F$-shrinkable} if there exists a $\gamma$-cover $\{F(U) :
U\in \mathcal{U}\}$ of zero-sets of $X$ with $F(U)\subseteq U$ for
some $U\in \mathcal{U}$ (\cite{os2}).

For a topological space $X$ we let $\Gamma_F\subseteq\Gamma$
denote the family of $\gamma_F$-shrinkable covers of $X$.

\begin{theorem}\label{th34}
For a Tychonoff space $X$ with $*\in \{1,fin\}$,
the following are equivalent:

\begin{enumerate}

\item TWO has a winning (Markov) strategy in $G_{*}(\Gamma_F,
\Omega)$ on $X$;

\item TWO has a winning (Markov) strategy in $G_{*}$
$(\Gamma_{\bf0},\Omega_{\bf0})$ on $C_p(X)$;

\item TWO has a winning (Markov) strategy in
$G_{*}(\mathcal{S},\Omega_{\bf0})$ on $C_p(X)$.

\end{enumerate}

\end{theorem}

\begin{proof}

$(1)\Rightarrow(2)$. For each $B\in \Gamma_{\bf 0}$ we define
$\mathcal{U}_n(B)=\{f^{\leftarrow}[(-\frac{1}{2^n}, \frac{1}{2^n})]: f\in
B\}$. To see that $\mathcal U_n(B)\in\Gamma_F$, let $x\in X$.
Since $B\in\Gamma_{\bf 0}$, $B\setminus[x;(-\frac{1}{2^{n+1}},\frac{1}{2^{n+1}})]$
is finite. It follows that for $f\in B\cap[x;(-\frac{1}{2^{n+1}},\frac{1}{2^{n+1}})]$,
\[
x\in f^{\leftarrow}\left[(-\frac{1}{2^{n+1}},\frac{1}{2^{n+1}})\right]\subseteq
f^{\leftarrow}\left[[-\frac{1}{2^{n+1}},\frac{1}{2^{n+1}}]\right]\subseteq
f^{\leftarrow}\left[(-\frac{1}{2^{n}},\frac{1}{2^{n}})\right]\
\]
and we have shown that
$\{f^{\leftarrow}[[-\frac{1}{2^{n+1}},\frac{1}{2^{n+1}}]]:f\in
B\}$ is a $\gamma$ cover by zero sets; therefore $\mathcal
U_n(B)\in\Gamma_F$.

Let $B_n\in\Gamma_{\bf0}$, and for $U\in\mathcal U_n(B_n)$ fix $f_{U,n}\in B_n$
such that $U=f_{U,n}^{\leftarrow}[(-\frac{1}{2^n},\frac{1}{2^n})]$.

If TWO $\uparrow G_*(\Gamma_F,\Omega)$ holds, then we may find a
winning strategy $\sigma$ that not only produces $\omega$ covers,
but produces covers such that every cofinite subset is an $\omega$
cover. To see this, partition any play by ONE into infinitely many
subplays and consider the strategy that applies the known winning
strategy to each subplay (the beginnings of which are cofinal in
$\omega$).

Now let $\tau(\tuple{B_0,\dots,B_n})=\{f_{U,n}: U\in
\sigma(\tuple{\mathcal{U}_0(B_0),\dots,\mathcal{U}_n(B_n)})\}$.
(Note here that the cardinalities of moves made by $\sigma$
are no greater than the cardinalities produced by
$\tau$, so this proof applies to both $G_1$ and $G_{fin}$.)
We claim that ${\bf 0}\in \overline{\bigcup\limits_{n<\omega}
\tau(\tuple{B_0,...,B_n})}$.

To see this, let $G\in [X]^{<\omega}$ and $\epsilon>0$.
Then choose $n<\omega$ such that $\frac{1}{2^{n}}<\epsilon$
and $G\subseteq U$ for some $U\in\sigma(\tuple{
\mathcal U_0(B_0),\dots,\mathcal U_n(B_n)})$.
Then \[
G\subseteq f_{U,n}^{\leftarrow}\left[(-\frac{1}{2^n},\frac{1}{2^n})\right]
\subseteq f_{U,n}^{\leftarrow}[(-\epsilon,\epsilon)]
\]
demonstrates that $f_{U,n}\in\tau(\tuple{B_0,\dots,B_n})\cap
[G;(-\epsilon,\epsilon)]$, verifying our claim.

If TWO ${\markwin} G_*(\Gamma_F,\Omega)$ holds, then we may again
assume we have a witnessing strategy $\sigma$ producing omega
covers such that every cofinite subset is an $\omega$-cover, for
the same reason as above.

Now let $\tau(B_n,n)=\{f_{U,n}: U\in
\sigma(\mathcal U_n(B_n),n)\}$.
(Note again here that the cardinality of $\sigma$ matches the cardinality
of $\tau$, so this proof applies to both $G_1$ and $G_{fin}$.)
We claim that ${\bf 0}\in \overline{\bigcup\limits_{n<\omega}
\tau(B_n,n)}$.

To see this, let $G\in [X]^{<\omega}$ and $\epsilon>0$.
Then choose $n<\omega$ such that $\frac{1}{2^{n}}<\epsilon$
and $G\subseteq U$ for some $U\in\sigma(\mathcal U_n(B_n),n)$.
Then \[
G\subseteq f_{U,n}^{\leftarrow}\left[(-\frac{1}{2^n},\frac{1}{2^n})\right]
\subseteq f_{U,n}^{\leftarrow}[(-\epsilon,\epsilon)]
\]
demonstrates that $f_{U,n}\in\tau(B_n,n)\cap
[G;(-\epsilon,\epsilon)]$, verifying our claim.

$(2)\Rightarrow(3)$. For each $S\in\mathcal S$, select
$G_S\subseteq S$ such that $\lim G={\bf0}$. Given a strategy for TWO in
$G_*(\Gamma_{\bf0},\Omega_{\bf0})$, TWO's strategy for
$G_*(\mathcal S,\Omega_{\bf0})$ simply substitutes each $S\in\mathcal S$
with $G_S$.

$(3)\Rightarrow(1)$. For each $\mathcal{U}\in \Gamma_F$ define
$S(\mathcal{U})=\{f\in C(X): f\upharpoonright (X\setminus U)\equiv
1$
 for some $U\in \mathcal{U}\}$. By \cite[Lemma 6.5]{os2}, $S(\mathcal{U})$ is
sequentially dense in $C_p(X)$.
Let $\mathcal U_n\in\Gamma_F$, and for each $f\in S(\mathcal U_n)$ choose
$U_{f,n}\in\mathcal U_n$ where $f\upharpoonright(X\setminus U_{f,n})\equiv 1$.

So let $\sigma$ witness TWO $\uparrow
G_{*}(\mathcal{S},\Omega_{\bf0})$, so ${\bf 0}\in
\overline{\bigcup\limits_{n<\omega}
\sigma(\tuple{S(\mathcal{U}_0),..., S(\mathcal{U}_n)})}$. We then
define $\tau(\tuple{\mathcal{U}_0,...,\mathcal{U}_n})=\{U_{f,n}:
f\in \sigma(\tuple{S(\mathcal{U}_0),..., S(\mathcal{U}_n)})\}$.
Let $F\in[X]^{<\omega}$, so we may choose $n\in\omega$ such that
$f\in\sigma(\tuple{S(\mathcal U_0),\dots,S(\mathcal U_n)})\cap
[F;(-1/2,1/2)]$. Then as $f\upharpoonright F$ cannot map to $1$,
$F\subseteq U_{f,n}$. Therefore $\tau$ produces $\omega$-covers.

Finally, let $\sigma$ witness TWO ${\markwin}
G_{*}(\mathcal{S},\Omega_{\bf0})$, so ${\bf 0}\in
\overline{\bigcup\limits_{n<\omega} \sigma(S(\mathcal{U}_n),n)}$.
We then define $\tau(\mathcal{U}_n,n)=\{U_{f,n}: f\in
\sigma(S(\mathcal{U}_n),n)\}$. Let $F\in[X]^{<\omega}$, so we may
choose $n\in\omega$ such that $f\in\sigma(S(\mathcal U_n),n)\cap
[F;(-1/2,1/2)]$. Then as $f\upharpoonright F$ cannot map to $1$,
$F\subseteq U_{f,n}$. Therefore $\tau$ produces $\omega$-covers.
\end{proof}

\begin{corollary}\label{th34}
Let $X$ be a Tychonoff space with a coarser second countably
topology  and $*\in \{1,fin\}$. The following assertions are
equivalent:

\begin{enumerate}
\item TWO has a winning (Markov) strategy in $G_{*}(\Gamma_F,
\Omega)$ on $X$;

\item TWO has a winning (Markov) strategy in $G_{*}$
$(\Gamma_{\bf0},\Omega_{\bf0})$ on $C_p(X)$;

\item TWO has a winning (Markov) strategy in
$G_{*}(\mathcal{S},\Omega_{\bf0})$ on $C_p(X)$.

\item TWO has a winning (Markov) strategy in
$G_{*}(\mathcal{S},\mathcal{D})$ on $C_p(X)$;

\end{enumerate}

\end{corollary}
\begin{proof}

By Theorems \ref{th5} and \ref{th31}, items (3) and (4) are
equivalent.
\end{proof}

\bibliographystyle{model1a-num-names}
\bibliography{<your-bib-database>}




\end{document}